\documentclass[a4paper,11pt]{amsart}

\usepackage{amsmath}
\usepackage{eucal}
\usepackage{amssymb}
\usepackage{cite}
\usepackage{xypic}

\newtheorem{theorem}{Theorem}[section]
\newtheorem{proposition}[theorem]{Proposition}

\newtheorem{lemma}[theorem]{Lemma}
\newtheorem*{conjecture}{Conjecture}
\theoremstyle{definition}

\theoremstyle{remark}
\newtheorem{remark}{Remark}[section]

\newcommand{\z}[1]{\frac\partial{\partial z_{#1}}}
\newcommand{\dz}[1]{dz_{#1}}
\newcommand{\dzb}[1]{d\overline z_{#1}}
\renewcommand{\epsilon}{\varepsilon}
\renewcommand{\tilde}{\widetilde}
\renewcommand{\bar}{\overline}
\newcommand{\x}[1]{\frac\partial{\partial \xi_{#1}}}
\newcommand{\rank}{\operatorname{rank}}
\newcommand{\tr}{\operatorname{tr}}

\begin{document}
\bibliographystyle{siam}
\title{Smooth metrics on jet bundles and applications}
\author{Simone Diverio}
\address{Simone Diverio --- Istituto \lq\lq Guido Castelnuovo\rq\rq{} \\ SAPIENZA Universit\`a di Roma \\ Piazzale Aldo Moro 2 \\ 00185 Roma.}
\email{diverio@mat.uniroma1.it} 

\subjclass[2000]{Primary 53B35; Secondary 14J29}

\begin{abstract}
Following a suggestion made by J.-P. Demailly, for each $k\ge 1$, we endow, by an induction process, the $k$-th (anti)tautological line bundle $\mathcal O_{X_k}(1)$ of an arbitrary complex directed manifold $(X,V)$ with a natural smooth hermitian metric. Then, we compute recursively the Chern curvature form for this metric, and we show that it depends (asymptotically -- in a sense to be specified later) only on the curvature of $V$ and on the structure of the fibration $X_k\to~X$. When $X$ is a surface and $V=T_X$, we give explicit formulae to write down the above curvature as a product of matrices. As an application, we obtain a new proof of the existence of global invariant jet differentials vanishing on an ample divisor, for $X$ a minimal surface of general type whose Chern classes satisfy certain inequalities, without using a strong vanishing theorem \cite{Bog77} of Bogomolov.
\end{abstract}

\maketitle
\tableofcontents

\section{Introduction}\label{introduction}

In \cite{GG80}, Green and Griffiths showed, among other things, that if $X$ is an algebraic surface of general type, then there exist $m\gg k\gg 1$, such that $H^0(X,\mathcal J_{k,m}T^*X)\ne 0$, where $\mathcal J_{k,m}T^*X$ is the bundle of jet differentials of order $k$ and degree $m$.  
Their proof relies on an asymptotic computation of the Euler characteristic $\chi(\mathcal J_{k,m}T^*X)$ (which has been possible thanks to the full knowledge of the composition series of this bundle) together with a powerful vanishing theorem of Bogomolov \cite{Bog77}. 

More precisely, if $X$ is a $n$-dimensional smooth projective variety, and $\mathcal J_{k,m}T^*X\to X$ the bundle of jet differentials of order $k$ and weighted degree $m$, they get the following asymptotic estimate for the holomorphic Euler characteristic:
$$
\begin{aligned}
\chi(\mathcal J_{k,m}T^*_X)=& \frac{m^{(k+1)n-1}}{(k!)^n((k+1)n-1)!} \\
&\times\biggl(\frac{(-1)^n}{n!}c_1(X)^n(\log k)^n+O\bigl((\log k)^{n-1}\bigl)\biggr)+O\bigl(m^{(k+1)n-2}\bigr).
\end{aligned}
$$
In particular, if $X$ is a surface of general type, then the Bogomolov vanishing theorem applies and, having cancelled the $h^2$ term by Serre's duality, they get a positive lower bound for $h^0(X,\mathcal J_{k,m}T^*_X)$ when $m\gg k\gg 1$.

Nowadays, there are no general results about the existence of global \emph{invariant} jet differentials on a surface of general type neither, of course, for varieties of general type in arbitrary dimension.  

Nevertheless, thanks to a beautiful and relatively simple argument of Demailly \cite{Dem07}, their existence should potentially lead to solve the following celebrated conjecture.

\begin{conjecture}[Green and Griffiths \cite{GG80}, Lang]
Let $X$ be an algebraic variety of general type. Then there exist a proper algebraic sub-variety $Y\subsetneq X$ such that every non-constant holomorphic entire curve $f\colon\mathbb C\to X$, has image $f(\mathbb C)$ contained in $Y$.
\end{conjecture}

A positive answer to this conjecture in dimension $2$ has been given by McQuillan in \cite{McQ98}, when the second Segre number $c_1(X)^2-c_2(X)$ of $X$ is positive (this hypothesis ensures the existence of an algebraic (multi)fo\-li\-ation on $X$, whose parabolic leaves are shown to be algebraically degenerate: this is the very deep and difficult part of the proof). 

\subsection{Main ideas and statement of the results}

Let $(X,V)$ be a complex directed manifold (for precise definitions see next section) with $\dim X=n$ and $2\le \rank V=r\le n$. Let $\omega$ be a hermitian metric on $V$. Such a metric naturally induces a smooth hermitian metric on the tautological line bundle $\mathcal O_{\tilde X}(-1)$ on the projectivized bundle of line of $V$. 

Now, the Chern curvature of its dual $\mathcal O_{\tilde X}(1)$, is a $(1,1)$-form on $\tilde X$ whose restriction to the fiber over a point $x\in X$ coincides with the Fubini-Study metric of $P(V_x)$ with respect to $\omega|_{V_x}$. Thus, it is positive in the fibers direction. Next, consider the pullback $\pi^*\omega$ on $\tilde X$: this is a $(1,1)$-form which is zero in the fibers direction and, of course, positive in the base direction.

If $X$ is compact so is $\tilde X$ and hence, for all $\varepsilon>0$ small enough, the restriction to $\tilde V$ of the $(1,1)$-form given by
$$
\pi^*\omega+\varepsilon^2\,\Theta(\mathcal O_{\tilde X}(1))
$$
gives rise to a hermitian metric on $\tilde V$. Moreover, this metric depends on two derivatives of the metric $\omega$. 

Of course, we can repeat this process for the compact directed manifold $(\tilde X,\tilde V)$, and by induction, for each $k\ge 1$ for the tower of projectivized bundles $(X_{k},V_k)$. \emph{A priori}, the hermitian metric we obtain in this fashion on $\mathcal O_{X_k}(-1)$, depends on $2k$ derivatives of the starting metric $\omega$ and on the choice of $\epsilon^{(k)}=(\varepsilon_1,\dots,\varepsilon_{k-1})$.

However, from a philosophical point of view, we would like to avoid the dependence on the last $2k-2$ derivatives of $\omega$, since the relevant geometrical data for $X$ lies in the first two derivatives of $\omega$, namely on its Chern curvature. Here comes Demailly's suggestion: as $\epsilon^{(k)}$ has to be small enough, it is quite natural to look for an asymptotic expression of the Chern curvature of the metric on $\mathcal O_{X_k}(-1)$ we have constructed, when $\epsilon^{(k)}$ tends to zero: this idea is developed in our first theorem.

\begin{theorem}\label{4.1}
The vector bundle $V_k$ can be endowed inductively with a smooth hermitian metric
$$
\omega^{(k)}=\left.\bigl(\pi_{k}^*\omega^{(k-1)}+\varepsilon_k^2\,\Theta(\mathcal O_{X_{k}}(1))\bigr)\right|_{V_k},
$$
where the metric on $\mathcal O_{X_k}(1)$ is induced by $\omega^{(k-1)}$, depending on $k-1$ positive real numbers $\varepsilon^{(k)}=(\varepsilon_1,\dots,\varepsilon_{k-1})$, such that the asymptotic of its Chern curvature with respect to this metric depends only on the curvature of $V$ and on the (universal) structure of the fibration $X_k\to X$, as $\varepsilon^{(k)}\to 0$.
\end{theorem}

As a byproduct of the proof of the above theorem, we also obtain induction formulae for an explicit expression of the curvature in terms of the curvature coefficients of $V$. These formulae, which are quite difficult to handle in higher dimension, are reasonably simple for $X$ a smooth surface: in this case, it turns out that the curvature coefficients of $\mathcal O_{X_k}(-1)$ are given by a sequence of products of $2\times 2$ real matrices. 

A general remark in analytic geometry is that the existence of global sections of a hermitian line bundle is strictly correlated with the positivity properties of its Chern curvature form. One of the countless correlations, is given by the theory of Demailly's holomorphic Morse inequalities \cite{Dem85}. We summarize his main result here below.

\subsubsection{Holomorphic Morse inequalities}

Let $X$ be a compact K\"ahler manifold of dimension $n$, $E$ a holomorphic vector bundle of rank $r$ and $L$ a line bundle over $X$. If $L$ is equipped with a smooth metric of curvature form $\Theta(L)$, we define the $q$-index set of $L$ to be the open subset
$$
X(q,L)=\left\{x\in X\mid\text{$i\,\Theta(L)$ has $\begin{matrix} \text{$q$ negative eigenvalues} \\ \text{$n-q$ positive eigenvalues}\end{matrix}$}\right\},
$$
for $q=0,\dots,n$. Hence $X$ admits a partition $X=\Delta\cup\bigcup_{q=0}^nX(q,L)$, where $\Delta=\{x\in X\mid\det(i\,\Theta(L))=0\}$ is the degeneracy set. We also introduce
$$
X(\le q,L)\overset{\text{\rm def}}=\bigcup_{j=0}^qX(j,L).
$$
It was shown by Demailly in \cite{Dem85}, that the partial alternating sums of the dimension of the cohomology groups of tensor powers of $L$ with values in $E$ satisy the following asymptotic \emph{strong Morse inequalities} as $k\to+\infty$:
$$
\sum_{j=0}^q(-1)^{q-j}h^j(X,L^{\otimes k}\otimes E)\le r\frac{k^n}{n!}\int_{X(\le q,L)}(-1)^q\biggl(\frac{i}{2\pi}\,\Theta(L)\biggr)^n+O(k^{n-1}).
$$
In particular, if
$$
\int_{X(\le 1,L)}\biggl(\frac{i}{2\pi}\,\Theta(L)\biggr)^n>0,
$$
then some high power of $L$ twisted by $E$ has a (many, in fact) nonzero section. 

\vspace{0,4cm}

The idea is now to apply holomorphic Morse inequalities to the anti-tautological line bundle $\mathcal O_{X_k}(1)$ together with the asymptotic hermitian metric constructed above, to find global sections of invariant $k$-jet differentials on a surface $X$: we shall deal with the absolute case $V=T_X$. Our first geometrical hypothesis is to suppose $X$ to be K\"ahler-Einstein, that is with ample canonical bundle. Nevertheless, standard arguments coming from the theory of Monge-Amp\`ere equations, will show that we just need to assume $X$ to be minimal and of general type, that is $K_X$ big and numerically effective. Finally, once sections are found, we can drop the hypothesis of nefness, since the dimension of the space of global section of jet differentials is a birational invariant (see, for instance, \cite{GG80} and \cite{Dem97}).

For each $k\ge 1$, in $\mathbb R^k$ define the closed convex cone $\mathfrak N=\{\mathbf a=(a_1,\dots,a_k)\in\mathbb R^k\mid a_j\ge2\sum_{\ell={j+1}}^ka_\ell$ for all $j=1,\dots,k-1$ and $a_k\ge 0\}$. For $X$ a smooth compact surface, set
$$
\mathcal O_{X_k}(\mathbf a)^{k+2}=F_k(\mathbf a)\,c_1(X)^2-G_k(\mathbf a)\,c_2(X)
$$
(see next section for the definition of the weighted line bundle $\mathcal O_{X_k}(\mathbf a)$) and 
$$
m_k=\sup_{\mathbf a\in\mathfrak N\setminus\Sigma_k}\frac{F_k(\mathbf a)}{G_k(\mathbf a)},
$$
where $\Sigma_k$ is the zero locus of $G_k$. Finally, call $m_\infty$ the supremum of the sequence $\{m_k\}$.

\begin{theorem}\label{4.2} 
Notations as above, the two following facts can occur: 

\noindent
either
\begin{itemize}
\item there exists a $k_0\ge 1$ such that for every surface $X$ of general type, $\mathcal O_{X_{k_0}}(1)$ is big,
\end{itemize}
or
\begin{itemize}
\item the sequence $\{m_k\}$ is positive non-decreasing and for $X$ a surface of general type, there exists a positive integer $k$ such that $\mathcal O_{X_k}(1)$ is big as soon as $m_\infty>c_2(\widehat X)/c_1(\widehat X)^2$, where $\widehat X$ is the minimal model of $X$. 
\end{itemize}
\end{theorem} 

As a corollary, we obtain the existence of low order jet differentials, for $X$ a minimal surface of general type whose Chern classes satisfy certain inequalities. This will be done in \S\ref{lb}.

\subsection*{Acknowledgments} We would like to thank Stefano Trapani for his suggestions, comments and remarks which have finally lead to the present version of Theorem 1.2. 

\section{Projectivized jet bundles}

Let $(X,V)$ be a complex directed manifold, that is a pair consisting in a smooth complex manifold $X$ and a holomorphic (non necessarily integrable) subbundle $V\subset T_X$ of the tangent bundle.
Set $\widetilde X = P(V)$. Here, $P(V)$ is the projectivized bundle of lines of $V$ and there is a natural projection $\pi\colon\widetilde X\to X$; moreover, if $\dim X=n$, then $\dim\widetilde X=n+r-1$, if $\rank V=r$. On $\widetilde X$, we consider the tautological line bundle $\mathcal O_{\widetilde X}(-1)\subset\pi^* V$ which is defined fiberwise as 
$$
\mathcal O_{\widetilde X}(-1)_{(x,[v])}\overset{\text{\rm def}}=\mathbb C\,v,
$$
for $(x,[v])\in\widetilde X$, with $x\in X$ and $v\in V_x\setminus\{0\}$. Next, set $\widetilde V=\pi_*^{-1}\mathcal O_{\widetilde X}(-1)$, where $\pi_*\colon T_{\widetilde X}\to\pi^* T_X$ is the differential of the projection: this is a holomorphic subbundle of $T_{\widetilde X}$ of rank $r$, so that we obtain in this fashion a new directed manifold $(\widetilde X,\widetilde V)$.

Now, we start the inductive process in the directed manifold category by setting
$$
(X_{0},V_{0})=(X,V),\quad (X_{k},V_{k})=(\tilde X_{k-1},\tilde V_{k-1}).
$$
In other words, $(X_k,V_k)$ is obtained from $(X,V)$ by iterating $k$ times the projectivization construction $(X,V)\mapsto (\tilde X,\tilde V)$ described above.

In this process, the rank of $V_k$ remains constantly equal to $r$ while the dimension of $X_k$ growths linearly with $k$: $\dim X_k=n+k(r-1)$. Let us call $\pi_k\colon X_k\to X_{k-1}$ the natural projection. Then we have, as before, a tautological line bundle $\mathcal O_{X_k}(-1)\subset\pi_k^* V_{k-1}$ over $X_k$ which fits into short exact sequences
\begin{equation}\label{ses1}
\xymatrix{
0\ar[r] & T_{X_k/ X_{k-1}}\ar[r] &V_k \ar[r]^{\!\!\!\!\!\!\!\!(\pi_k)_*} & \mathcal O_{X_k}(-1)\ar[r] & 0}
\end{equation}
and
\begin{equation}\label{ses2}
\xymatrix{
0\ar[r] &\mathcal O_{X_k}\ar[r] &\pi_k^* V_{k-1}\otimes\mathcal O_{X_k}(1)\ar[r] & T_{X_k/X_{k-1}}\ar[r] & 0,}
\end{equation}
where $T_{X_k/ X_{k-1}}=\ker(\pi_k)_*$ is the relative tangent bundle.

More generally, if $\mathbf a=(a_1,\dots,a_k)\in\mathbb Z^k$ is a weight, we can form the line bundle $\mathcal O_{X_k}(\mathbf a)$ by setting
$$
\mathcal O_{X_k}(\mathbf a)=\bigotimes_{j=1}^k\pi_{j,k}^*\mathcal O_{X_j}(a_j).
$$
We shall see later how, for appropriate choices of $\mathbf a$, one can obtain relatively positive line bundles $\mathcal O_{X_k}(\mathbf a)$ which, moreover, admit a non-trivial morphism to $\mathcal O_{X_k}(a_1+\cdots+a_k)$ (for the last assertion see for example \cite{Dem97}). In particular, sections of $\mathcal O_{X_k}(\mathbf a)$ for a suitable choice of $\mathbf a$ give rise to sections of $\mathcal O_{X_k}(m)$ for some large $m$.
 
\section{From $(X,V)$ to $(\tilde X,\tilde V)$}

Let $(X,V)$ be a compact directed manifold of complex dimension $n$ and $\rank V=r$. In this section, given a hermitian metric $\omega$ on $V$, we construct a (family of) metric on $\tilde V$ depending on a \lq\lq small\rq\rq{} positive constant $\varepsilon$, and we compute the curvature of $\tilde V$ with respect to this metric, letting $\varepsilon$ tend to zero.

So, fix a hermitian metric $\omega$ on $V$, a point $x_0\in X$ and a unit vector $v_0\in V_{x_0}$ with respect to $\omega$. Then there exist coordinates $(z_1,\dots,z_n)$ centered at $x_0$ and a holomorphic normal local frame $e_1,\dots,e_r$ for $V$ such that $e_r(x_0)=v_0$ and
$$
\omega(e_\lambda,e_\mu)=\delta_{\lambda\mu}-\sum_{j,k=1}^n c_{jk\lambda\mu}\,z_j\overline z_k+O(|z|^3).
$$
Remark that, as $V$ is a holomorphic subbundle of the holomorphic tangent space of $X$, then there exists a holomorphic matrix $(g_{i\lambda}(z))$ such that $e_\lambda(z)=\sum_{i=1}^n g_{i\lambda}(z)\frac\partial{\partial z_i}$.

Moreover, the Chern curvature at $x_0$ of $V$ is expressed by
$$
\Theta(V)_{x_0}=\sum_{j,k=1}^n\sum_{\lambda,\mu=1}^r c_{jk\lambda\mu}\,\dz j\wedge\dzb k\otimes e_\lambda^*\otimes e_\mu.
$$
Now consider the projectivized bundle $\pi\colon P(V)=\widetilde X\to X$ of lines in $V$: its points can be seen as pairs $(x,[v])$ where $x\in X$, $v\in V_{x}\setminus\{0\}$ and $[v]=\mathbb C v$. In a neighborhood of $(x_0,[v_0])\in\tilde X$ we have local holomorphic coordinates given by $(z,\xi_1,\dots,\xi_{r-1})$ where $\xi$ corresponds to the direction $[\xi_1 e_1(z)+\cdots+\xi_{r-1} e_{r-1}(z)+e_r(z)]$ in $V_z$.

On $\tilde X$ we have a tautological line bundle $\mathcal O_{\tilde X}(-1)\subset\pi^*V$ such that the fiber over $(x,[v])$ is simply $[v]$: then $\mathcal O_{\tilde X}(-1)\subset\pi^*V$ inherits a metric from $V$ in such a way that its local non vanishing section $\eta(z,\xi)=\xi_1 e_1(z)+\cdots+\xi_{r-1} e_{r-1}(z)+e_r(z)$ has squared length
$$
\begin{aligned}
|\eta|_\omega^2= & 1+|\xi|^2-\sum_{j,k,\lambda,\mu}c_{jk\lambda\mu}\,z_j\overline z_k\xi_\lambda\bar\xi_\mu-\sum_{j,k,\lambda}c_{jk\lambda r}\,z_j\overline z_k\xi_\lambda \\
& -\sum_{j,k,\mu}c_{jkr\mu}\,z_j\overline z_k\bar\xi_\mu-\sum_{j,k}c_{jkrr}\,z_j\overline z_k+O(|z|^3).
\end{aligned}
$$
So we have
$$
\begin{aligned}
\bar\partial|\eta|_\omega^2=&\sum_{\mu}\xi_\mu\,d\bar\xi_\mu-\sum_{j,k,\lambda,\mu}c_{jk\lambda\mu}z_j\overline z_k\xi_\lambda\,d\bar\xi_\mu\\
 &-\sum_{j,k}c_{jkrr}z_j\,d\overline z_k+O((|z|+|\xi|)^2|dz|+|z|^2|d\xi|), \\
 \partial|\eta|_\omega^2=&\sum_{\lambda}\bar\xi_\lambda\,d\xi_\lambda-\sum_{j,k,\lambda,\mu}c_{jk\lambda\mu}z_j\overline z_k\bar\xi_\mu\,d\xi_\lambda\\
& -\sum_{j,k}c_{jkrr}\bar z_k\,d z_j+O((|z|+|\xi|)^2|dz|+|z|^2|d\xi|), \\
 \partial\bar\partial|\eta|_\omega^2=&
\sum_{\lambda}d\xi_\lambda\wedge d\bar\xi_\lambda-\sum_{j,k,\lambda,\mu}c_{jk\lambda\mu}z_j\overline z_k\,d\xi_\lambda\wedge d\bar\xi_\mu 
-\sum_{j,k=1}^n c_{jkrr}\,\dz j\wedge\dzb k\\&+O((|z|+|\xi|)|dz|^2+|z|\,|dz|\,|d\xi|+(|z|+|\xi|)^3|d\xi|^2),
\end{aligned}
$$
where all the summations here are taken with $j,k=1,\dots,n$ and $\lambda,\mu=1,\dots,r-1$. We remark that inside the $O$'s there are hidden terms which are useless for our further computations. 
We finally obtain
$$
\begin{aligned}
\Theta(\mathcal O_{\tilde X}(1)) & = \partial\bar\partial\log|\eta|_\omega^2  = -\frac 1{|\eta|_\omega^4}\partial|\eta|_\omega^2\wedge\bar\partial|\eta|_\omega^2
+\frac 1{|\eta|_\omega^2}\partial\bar\partial|\eta|_\omega^2 \\
& =\sum_{\lambda,\mu}\biggl(-\xi_\mu\bar\xi_\lambda-\sum_{j,k}c_{jk\lambda\mu}z_j\bar z_k \\
& \quad +\delta_{\lambda\mu}\bigl(1-|\xi|^2+\sum_{j,k}c_{jkrr}z_j\bar z_k)\bigr)\biggr)\,d\xi_\lambda\wedge d\bar\xi_\mu \\
& \quad -\sum_{j,k} c_{jkrr}\,\dz j\wedge\dzb k\\&\quad+O((|z|+|\xi|)|dz|^2+|z|\,|dz|\,|d\xi|+(|z|+|\xi|)^3|d\xi|^2).
\end{aligned}
$$
So we get in particular
$$
\Theta(\mathcal O_{\tilde X}(1))_{(x_0,[v_0])}=
\sum_{\lambda=1}^{r-1}d\xi_\lambda\wedge d\bar\xi_\lambda-\sum_{j,k=1}^n c_{jkrr}\,\dz j\wedge\dzb k,
$$
which shows that
$$
\Theta(\mathcal O_{\tilde X}(1))_{(x_0,[v_0])}=|\bullet|_{\text{FS}}^2-\theta_{V,x_0}(\bullet\otimes v_0,\bullet\otimes v_0),
$$
where FS denotes the Fubini-Study metric along the vertical tangent space $\ker\pi_*$ and $\theta_{V,x_0}$ is the natural hermitian form on $T_X\otimes V$ corresponding to $i\,\Theta(V)$, at the point $x_0$.

Now consider the rank $r$ holomorphic subbundle $\tilde V$ of $T_{\tilde X}$ whose fiber over a point $(x,[v])$ is given by   
$$
\tilde V_{(x,[v])}=\{\tau\in T_{\tilde X}\mid\pi_*\tau\in\mathbb C v\}.
$$
To start with, let's consider the holomorphic local frame of $\tilde V$ given by $\x 1,\dots,\x{r-1},\tilde\eta$, where 
$$
\tilde\eta(z,\xi)=\sum_{i=1}^n\biggl(g_{ir}(z)+\sum_{\lambda=1}^{r-1} g_{i\lambda}(z)\xi_\lambda\biggr)\frac\partial{\partial z_i},
$$
so that $\tilde\eta$ formally is equal to $\eta$ but here, with a slight abuse of notations, the $\z i$ are regarded as tangent vector fields to $\tilde X$ (so, $\widetilde\eta$ actually means a lifting of $\eta$ from $\mathcal O_{\tilde X}(-1)\subset\pi^*V\subset\pi^* T_X$ to $T_{\tilde X}$, which admits $\pi^* T_X$ as a quotient).
For all sufficiently small $\epsilon>0$ we get an hermitian metric on $\tilde V$ by restricting $\tilde\omega_\epsilon=\pi^*\omega+\epsilon^2\,\Theta(\mathcal O_{\tilde X}(1))$ to $\tilde V$; at the point $(x_0,[v_0])=(0,0)$ we have
$$
\begin{aligned}
\tilde\omega_\epsilon\biggl(\x\lambda,\x\mu\biggr)= & \underbrace{\pi^*\omega\biggl(\x\lambda,\x\mu\biggr)}_{=0}
+\epsilon^2\,\Theta(\mathcal O_{\tilde X}(1))\biggl(\x\lambda,\x\mu\biggr) \\
= & \,\delta_{\lambda\mu}\epsilon^2,
\end{aligned}
$$
$$
\begin{aligned}
\tilde\omega_\epsilon\biggl(\x\lambda,\tilde\eta\biggr)= & \underbrace{\pi^*\omega\biggl(\x\lambda,\tilde\eta\biggr)}_{=0}
+\epsilon^2\,\Theta(\mathcal O_{\tilde X}(1))\biggl(\x\lambda,\tilde\eta\biggr) \\
= & 0,\quad\text{since $\Theta(\mathcal O_{\tilde X}(1))_{(x_0,[v_0])}\biggl(\x\lambda,\frac\partial{\partial z_i}\biggr)=0$}
\end{aligned}
$$
and
$$
\begin{aligned}
\tilde\omega_\epsilon(\tilde\eta,\tilde\eta)= & \pi^*\omega(\tilde\eta,\tilde\eta)
+\epsilon^2\,\Theta(\mathcal O_{\tilde X}(1))(\tilde\eta,\tilde\eta) \\
= & |\eta(0,0)|_\omega^2+O(\epsilon^2)=1+O(\epsilon^2).
\end{aligned}
$$
We now renormalize this local frame of $\tilde V$ by setting
$$
f_1=\frac 1\epsilon\x 1,\dots,f_{r-1}=\frac 1\epsilon\x{r-1},f_r=C_\epsilon\,\tilde\eta,
$$
where 
$$
C_\epsilon=\frac 1{\sqrt{\tilde\omega_\epsilon(\tilde\eta,\tilde\eta)}}=1+O(\epsilon).
$$
Then $(f_\lambda)$ is unitary at $(x_0,[v_0])$ with respect to $\tilde\omega_\epsilon$ and we have
$$
\tilde\omega_\epsilon(f_\lambda,f_\mu)=
\begin{cases}
-\xi_\mu\bar\xi_\lambda-\sum_{j,k}c_{jk\lambda\mu}z_j\bar z_k & \text{if $1\le\lambda,\mu\le r-1$} \\
+\delta_{\lambda\mu}(1-|\xi|^2+\sum_{j,k}c_{jkrr}z_j\bar z_k) & \\
 0 & \text{if $1\le\lambda\le r-1$ and $\mu=r$} \\
 & \text{or $1\le\mu\le r-1$ and $\lambda=r$} \\
 |\eta|_\omega^2 & \text{if $\lambda=\mu=r$},
\end{cases} 
$$
modulo $\epsilon$ and terms of order three in $z$ and $\xi$.

Next, we compute the curvature 
$$
\Theta(\tilde V)_{(x_0,[v_0])}=\sum_{j,k=1}^{n+r-1}\sum_{\lambda,\mu=1}^r \gamma_{jk\lambda\mu}\,\dz j\wedge\dzb k\otimes f_\lambda^*\otimes f_\mu
$$ 
for $\epsilon\to 0$, where we have set $z_{n+\lambda}=\xi_\lambda$.
Recall that for a hermitian vector bundle $E\to Y$, given a holomorphic trivialization, the curvature operator at a point $0\in Y$ is given by
$$
\Theta(E)_{0}=\bar\partial(\bar H^{\,-1}\partial\bar H)(0)=(\bar\partial\,\bar H^{\,-1})(0)\wedge(\partial\bar H)(0)+\bar H^{\,-1}(0)(\bar\partial\partial\bar H)(0),
$$
where $H$ is the hermitian matrix of hermitian products between the elements of the local frame.
If the local holomorphic frame is unitary in $0$, so that $H(0)=$ Id, observing that $0=\bar\partial(\bar H^{\,-1}\bar H)=(\bar\partial\,\bar H^{\,-1})(0)\,\bar H(0)+\bar H^{\,-1}(0)(\bar\partial\,\bar H)(0)$, we obtain
\begin{equation}\label{cur}
\Theta(E)_{0}=-\bar\partial\,\bar H(0)\wedge\partial\bar H(0)+\bar\partial\partial\bar H(0).
\end{equation}
Thus, in our case, it suffices to compute the part with second derivatives in (\ref{cur}) to get the following proposition.

\begin{proposition}
Notations as given, the Chern curvature of $\tilde V$ has the following expression:   
\begin{equation}\label{curV}
\begin{aligned}
\Theta(\tilde V)= & \sum_{\lambda,\mu=1}^{r-1}\left(d\xi_\mu\wedge d\bar\xi_\lambda+\sum_{j,k=1}^{n}c_{jk\lambda\mu}\,\dz j\wedge\dzb k \right.\\
&+\left.\delta_{\lambda\mu}\biggl(\sum_{\nu=1}^{r-1} d\xi_\nu\wedge d\bar\xi_\nu-\sum_{j,k=1}^n c_{jkrr}\,\dz j\wedge\dzb k\biggr)\right)\otimes f_\lambda^*\otimes f_\mu \\
& +\biggl(\sum_{j,k=1}^n c_{jkrr}\,\dz j\wedge\dzb k-\sum_{\nu=1}^{r-1} d\xi_\nu\wedge d\bar\xi_\nu\biggr)\otimes f_r^*\otimes f_r+O(\epsilon).
\end{aligned}
\end{equation}
In particular, we get the following identities modulo $\epsilon$:
$$
\begin{aligned}
& \gamma_{jk\lambda\mu}=
\begin{cases}
c_{jk\lambda\mu}-\delta_{\lambda\mu}c_{jkrr}& \text{if $1\le j,k\le n$ and $1\le\lambda,\mu\le r-1$} \\
\delta_{\lambda\mu}\delta_{jk}+\delta_{(j-n)\mu}\delta_{(k-n)\lambda} & \text{if $n+1\le j,k\le n+r-1$}\\
&  \text{and $1\le\lambda,\mu\le r-1$,}
\end{cases} \\
& \gamma_{jkrr}=
\begin{cases}
c_{jkrr} & \text{if $1\le j,k\le n$} \\
-1 & \text{if $n< j=k\le n+r-1$},
\end{cases}
\end{aligned}
$$
the remaining coefficients being zero.
\end{proposition}

\section{A special choice of coordinates and local frames}

We now pass to the tower of projective bundles $(X_k,V_k)$ over $(X,V)$. We recall that we simply set $(X,V)=(X_0,V_0)$ and, for all integer $k>0$, $(X_k,V_k)=(\tilde X_{k-1},\tilde V_{k-1})$ together with the projection $\pi_{k-1,k}\colon X_k\to X_{k-1}$ so that the total fibration is given by $\pi_{0,k}=\pi_{0,1}\circ\pi_{1,2}\circ\dots\circ\pi_{k-1,k}\colon X_k\to X$. 

For all $k$, we also have a tautological line bundle $\mathcal O_{X_k}(-1)$ and a metric $\omega^{(k)}=\omega^{(k)}(\epsilon_1,\dots,\epsilon_{k})$ on $V_k$, with the $\epsilon_l$'s positive and small enough, obtained recursively by setting $\omega^{(k)}=\bigl(\pi_{k-1,k}^*\omega^{(k-1)}+\epsilon_k^2\,\Theta(\mathcal O_{X_k}(1))\bigr)|_{V_k}$, $\omega^{(0)}=\omega$. 

To start with, fix a point $x_0\in X$, a $\omega$-unitary vector $v_0\in V$ and a holomorphic local normal frame $(e_\lambda^{(0)})$ for $(V,\omega)$ such that $e_r(x_0)=v_0$. 

\subsubsection*{First step}

On $X_1$, we have local holomorphic coordinates centered at $(x_0,[v_0])$ given by $(z,\xi^{(1)})$ where, $(z,\xi)\mapsto[\xi_1^{(1)} e_1^{(0)}(z)+\cdots+\xi_{r-1}^{(1)} e_{r-1}^{(0)}(z)+e_r^{(0)}(z)]\in P(V_z)$. Recall that we have, as before, a ``natural'' local section $\eta_1$ of $\mathcal O_{X_1}(-1)$ given by
$$
\eta_1(z,\xi^{(1)})=\xi_1^{(1)} e_1(z)+\cdots+\xi_{r-1}^{(1)} e_{r-1}(z)+e_r(z)
$$
and for all $\epsilon_1>0$ small enough, a holomorphic local frame $(f_\lambda^{(1)})$ for $V_1$ near $(x_0,[v_0])$ which is a $\omega^{(1)}$-unitary basis for ${V_1}_{(x_0,[v_0])}$.

Now, choose a $\omega^{(1)}$-unitary vector $v_1\in{V_1}_{(x_0,[v_0])}$ and a holomorphic local normal frame $(e_\lambda^{(1)})$ for $V_1$ such that $e_r^{(1)}(x_0,[v_0])=v_1$. Then there exist a unitary $r\times r$ matrix $U_1=(a_{\lambda\mu}^{(1)})$ such that at $(x_0,[v_0])$ we have
$$
f_\mu^{(1)}=\sum_{\lambda=1}^r a_{\lambda\mu}^{(1)}\,e_\lambda^{(1)}.
$$
So, if we call respectively $\gamma_{ij\lambda\mu}^{(1)}$ and $c_{ij\lambda\mu}^{(1)}$ the coefficients of curvature of $V_1$ at $(x_0,[v_0])$ with respect to the basis $(f_\lambda^{(1)})$ and $(e_\lambda^{(1)})$ we have
$$
c_{ij\lambda\mu}^{(1)}=\sum_{\alpha,\beta=1}^{r}\gamma_{ij\alpha\beta}^{(1)}\,\bar a_{\lambda\alpha}^{(1)}a_{\mu\beta}^{(1)},
$$
with $i,j=1,\dots,n+(r-1)$ and $\lambda,\mu=1,\dots,r$.

\subsubsection*{General step}

For the general case, suppose for all $\epsilon_1,\dots,\epsilon_{k-1}>0$ small enough we have built a system of holomorphic coordinates $(z,\xi^{(1)},\dots,\xi^{(k-1)})$ for $X_{k-1}$ and a holomorphic local normal frame $(e_\lambda^{(k-1)})$ for $(V_{k-1},\omega^{(k-1)})$, $k\ge 2$, such that $e_r^{(k-1)}(x_0,[v_0],\dots,[v_{k-2}])=v_{k-1}$ where $v_{k-1}$ is a $\omega^{(k-1)}$-unitary vector in ${V_{k-1}}_{(x_0,[v_0],\dots,[v_{k-2}])}$. Our procedure gives us also a holomorphic local frame $(f_\lambda^{(k-1)})$ for $V_{k-1}$ near $(x_0,[v_0],\dots,[v_{k-2}])$ which is a $\omega^{(k-1)}$-unitary basis for ${V_{k-1}}_{(x_0,[v_0],\dots,[v_{k-2}])}$ and a unitary $r\times r$ matrix $U_{k-1}=(a_{\lambda\mu}^{(k-1)})$ such that
$$
f_\mu^{(k-1)}=\sum_{\lambda=1}^r a_{\lambda\mu}^{(k-1)}\,e_\lambda^{(k-1)}.
$$
Then, we put holomorphic local coordinates $(z,\xi^{(1)},\dots,\xi^{(k)})$ on $X_k$ centered at the point $(x_0,[v_0],\dots,[v_{k-1}])$ where
$$
\begin{aligned}
(z,\xi^{(1)},\dots,\xi^{(k)})\mapsto[&\xi_1^{(k)} e_1^{(k-1)}(z,\dots,\xi^{(k-1)})+\cdots+\xi_{r-1}^{(k)} e_{r-1}^{(k-1)}(z,\dots,\xi^{(k-1)})\\
&+e_r^{(k-1)}(z,\dots,\xi^{(k-1)})]\in P({V_{k-1}}_{(z,\dots,\xi^{(k-1)})})
\end{aligned}
$$
and also 
$$
\begin{aligned}
\eta_k(z,\xi^{(1)},\dots,\xi^{(k)})=&\xi_1^{(k)} e_1^{(k-1)}(z,\dots,\xi^{(k-1)})+\cdots+\xi_{r-1}^{(k)} e_{r-1}^{(k-1)}(z,\dots,\xi^{(k-1)})\\
&+e_r^{(k-1)}(z,\dots,\xi^{(k-1)})
\end{aligned}
$$
is a local nonzero section of $\mathcal O_{X_k}(-1)$.

As we have already done, if we call
$$
f_\lambda^{(k)}=\frac 1{\epsilon_k}\frac\partial{\partial\xi_\lambda^{(k)}},\quad\lambda=1,\dots,r-1,\qquad f_r^{(k)}=C_{\epsilon_k}^{(k)}\,\tilde\eta_k,
$$
where $C_{\epsilon_k}^{(k)}=\frac 1{\sqrt{\omega^{(k)}(\tilde\eta_k,\tilde\eta_k)}}=1+O(\epsilon_k)$, then $(f_\lambda^{(k)})$ is a local holomorphic frame for $V_k$, unitary at $(x_0,[v_0],\dots,[v_{k-1}])$. We now fix a $\omega^{(k)}$-unitary vector $v_k\in {V_k}_{(x_0,[v_0],\dots,[v_{k-1}])}$ and choose a holomorphic local normal frame $(e_\lambda)^{(k)}$ for $(V_k,\omega^{(k)})$ such that $e_\lambda^{(k)}(x_0,[v_0],\dots,[v_{k-1}])=v_k$ and a $r\times r$ unitary matrix $U_k=(a_{\lambda\mu}^{(k)})$ such that $f_\mu^{(k)}=\sum_{\lambda=1}^r a_{\lambda\mu}^{(k)}\,e_\lambda^{(k)}$.

So, if we call respectively $\gamma_{jk\lambda\mu}^{(k)}$ and $c_{jk\lambda\mu}^{(k)}$ the coefficients of curvature of $V_k$ at $(x_0,[v_0],\dots,[v_{k-1}])$ with respect to the basis $(f_\lambda^{(k)})$ and $(e_\lambda^{(k)})$ we have
\begin{equation}\label{ind}
c_{ij\lambda\mu}^{(k)}=\sum_{\alpha,\beta=1}^{r}\gamma_{ij\alpha\beta}^{(k)}\,\bar a_{\lambda\alpha}^{(k)}a_{\mu\beta}^{(k)},
\end{equation}
with $i,j=1,\dots,n+k(r-1)$ and $\lambda,\mu=1,\dots,r$.

\section{Curvature of $\mathcal O_{X_k}(1)$ and proof of Theorem \ref{4.1}}

We now use (\ref{curV}) and (\ref{ind}) to get the induction formulae to derive an expression for the curvature of $\mathcal O_{X_k}(1)$, when $\epsilon^{(k)}=(\epsilon_1,\dots,\epsilon_{k-1})$ tends to zero.

We start by observing that (\ref{curV}) shows how $\gamma_{ij\lambda\mu}^{(s)}$ depends on $c_{lm\alpha\beta}^{(s-1)}$; we rewrite here the dependence modulo $\epsilon_s$:
\begin{equation}\label{recur}
\begin{aligned}
& \gamma_{ij\lambda\mu}^{(s)}=
\begin{cases}
c_{ij\lambda\mu}^{(s-1)}-\delta_{\lambda\mu}c_{ijrr}^{(s-1)} & \text{if $1\le i,j\le n+(s-1)(r-1)$}\\
&\text{and $1\le\lambda,\mu\le r-1$} \\
\delta_{jk}\delta_{\lambda\mu} +& \text{if $i,j\ge n+(s-1)(r-1)+1$,} \\
\delta_{(i-n-(s-1)(r-1))\mu}\delta_{(j-n-(s-1)(r-1))\lambda}& \text{$i,j\le n+s(r-1)$}\\
& \text{and $1\le\lambda,\mu\le r-1$,}
\end{cases} \\
& \gamma_{ijrr}^{(s)}=
\begin{cases}
c_{ijrr}^{(s-1)} & \text{if $1\le i,j\le n+(s-1)(r-1)$} \\
-1 & \text{if $n+(s-1)(r-1)+1 \le i=j\le n+s(r-1)$},
\end{cases}
\end{aligned}
\end{equation}
the remaining coefficients being zero. 
Recall also that, by (\ref{ind}),
$$
c_{ij\lambda\mu}^{(s)}=\sum_{\alpha,\beta=1}^{r}\gamma_{ij\alpha\beta}^{(s)}\,\bar a_{\lambda\alpha}^{(s)}a_{\mu\beta}^{(s)}.
$$
Now, we have
\begin{equation}\label{theta}
\Theta(\mathcal O_{X_k}(1))_{(x_0,[v_0],\dots,[v_{k-1}])}=\sum_{\lambda=1}^{r-1}d\xi_\lambda^{(k)}\wedge d\bar\xi_{\lambda}^{(k)}-\sum_{i,j=1}^{n+(k-1)(r-1)}c_{ijrr}^{(k-1)}\,\dz i\wedge\dzb j,
\end{equation}
where we have set $z_{n+(s-1)(r-1)+\lambda}=\xi_\lambda^{(s)}$, $\lambda=1,\dots,r-1$, and to get the expression of this curvature with respect to the coefficients of curvature of $V$ it suffices to perform the recursive substitutions (\ref{ind}) and (\ref{recur}) and to stop with $c_{ij\lambda\mu}^{(0)}=c_{ij\lambda\mu}$.

Thus, Theorem \ref{4.1} is proved.

\subsection{The case of surfaces}\label{surfnotation}

In the case $\rank V=\dim X=2$, we have a nice matrix representation of these formulae. First of all, note that in this case the identities (\ref{recur}) become much simpler:
$$
\begin{aligned}
& \gamma_{ij11}^{(s)}=
\begin{cases}
c_{ij11}^{(s-1)}-c_{ij22}^{(s-1)} & \text{if $1\le i,j\le s+1$} \\
2 & \text{if $i=j=s+2$} \\
\end{cases} \\
& \gamma_{ij22}^{(s)}=
\begin{cases}
c_{ij22}^{(s-1)} & \text{if $1\le i,j\le s+1$} \\
-1 & \text{if $ i=j=s+2$}.
\end{cases}
\end{aligned}
$$
Now, for each $s\ge 1$, let $v_s=v_s^1\,f_1^{(s)}+v_s^2\,f_2^{(s)}$, with $|v_s^1|^2+|v_s^2|^2=1$. Then we have $a_{21}^{(s)}=\bar v_s^1$ and $a_{22}^{(s)}=\bar v_s^2$ and so, for instance $a_{11}^{(s)}=-v_s^2$ and $a_{12}^{(s)}=v_s^1$ would work. It follows that
$$
c^{(s)}_{ij11}=\gamma_{ij11}^{(s)}|v_s^2|^2+\gamma_{ij22}^{(s)}|v_s^1|^2,\quad
c^{(s)}_{ij22}=\gamma_{ij11}^{(s)}|v_s^1|^2+\gamma_{ij22}^{(s)}|v_s^2|^2.
$$
So, if we set
$$
R_s=
\begin{pmatrix}
|v_s^2|^2 & |v_s^1|^2 \\
|v_s^1|^2 & |v_s^2|^2
\end{pmatrix},\quad
T=
\begin{pmatrix}
1 & -1 \\
0 & 1
\end{pmatrix},\quad
C_{ij}^{(s)}=
\begin{pmatrix}
c^{(s)}_{ij11} \\
c^{(s)}_{ij22}
\end{pmatrix}
$$
we have that
$$
C_{ij}^{(s)}=
\begin{cases}
R_s\cdot T\cdot R_{s-1}\cdot T\cdots R_1\cdot T\cdot C_{ij}^{(0)} & \text{if $1\le i,j\le 2$} \\
R_s\cdot T\cdot R_{s-1}\cdot T\cdots R_{i-2}\cdot T\cdot
\begin{pmatrix}
1 \\
-1
\end{pmatrix} & \text{if $3\le i=j\le s+2$}
\end{cases}
$$ 
and we are interested in the second element of the vector $C_{ij}^{(k-1)}$: in fact, in the surface absolute case, formula (\ref{theta}) can be rewritten in the form

\begin{equation}\label{generalformula}
\begin{aligned}
\Theta(\mathcal O_{X_k}(1))= \quad& d\xi^{(k)}\wedge d\bar\xi^{(k)}-\sum_{s=3}^{k+1}c_{ss22}^{(k-1)}\,d\xi^{(s-2)}\wedge d\bar\xi^{(s-2)} \\
& -\sum_{i,j=1}^{2}c_{ij22}^{(k-1)}\,\dz i\wedge\dzb j.
\end{aligned}
\end{equation}

We shall see in the next sections how this explicit formulae can be use to compute Morse-type integrals, in order to obtain the existence of nonzero global section of the bundle of invariant jet differentials.

\section{Holomorphic Morse inequalities for jets}

Let $X$ be a smooth surface and $V=T_X$. From now on we will suppose that $K_X$ is ample, so that we can take as a metric on $X$ the K\"ahler-Einstein one, and we will work always modulo $\varepsilon^{k}$ (this will be possible thanks to Lebesgue's dominated convergence theorem). 

\subsection{The K\"ahler-Einstein assumption}\label{D}

So, let $K_X$ be ample. Then we have a unique hermitian metric $\omega$ on $T_X$, such that $\operatorname{Ricci}(\omega)=-\omega$ and, for this metric,  
$$
\operatorname{Vol}_\omega (X)=\frac{\pi^2}{2}c_1^2(X)>0,
$$  
where
$$
\operatorname{Vol}_\omega(X)\overset{\text{\rm def}}=\int_X\frac{\omega^2}{2!}.
$$
Now, consider the two hermitian matrices $(c_{ij11})$ and $(c_{ij22})$. The K\"ahler-Einstein assumption implies that
$$
(c_{ij11})+(c_{ij22})=(-\delta_{ij})
$$
and so they are simultaneously diagonalizable. Let 
$$
\begin{pmatrix}
\lambda & 0 \\
0 & \mu 
\end{pmatrix}
$$
be a diagonal form for $(c_{ij11})$. Then 
$$
\lambda+\mu=c_{1111}+c_{2211}=c_{1111}+c_{1122}=-1
$$
thanks to the K\"ahler symmetries. If the eigenvalues of $(c_{ij22})$ are $\lambda',\mu'$ then  $\lambda'+\lambda=\mu+\mu'=-1$ thus a diagonal form for $(c_{ij22})$ is
$$  
\begin{pmatrix}
\mu & 0 \\
0 & \lambda 
\end{pmatrix}.
$$
As a consequence, for $\alpha,\beta\in\mathbb C$, the eigenvalues of the matrix $\alpha(c_{ij11})+\beta(c_{ij22})$ are $\alpha\lambda+\beta\mu$ and $\alpha\mu+\beta\lambda$ and so
\begin{equation}\label{det}
\begin{aligned}
\det(\alpha(c_{ij11})+\beta(c_{ij22})) &= (\alpha\lambda+\beta\mu)(\alpha\mu+\beta\lambda) \\
&= \alpha\beta(\lambda^2+\mu^2)+\lambda\mu(\alpha^2+\beta^2) \\
&=\alpha\beta[(\lambda+\mu)^2-2\lambda\mu]+\lambda\mu(\alpha^2+\beta^2) \\
&=\alpha\beta+\lambda\mu(\alpha-\beta)^2
\end{aligned}
\end{equation}
and
\begin{equation}\label{tr}
\begin{aligned}
\text{\rm tr}(\alpha(c_{ij11})+\beta(c_{ij22})) &= (\alpha\lambda+\beta\mu)+(\alpha\mu+\beta\lambda) \\
&= (\alpha+\beta)(\lambda+\mu)\\
&=-(\alpha+\beta).
\end{aligned}
\end{equation}

For $k=1$, the curvature of $\mathcal O_{X_1}(1)$ is simply
$$
\Theta(\mathcal O_{X_1}(1))=d\xi^{(1)}\wedge d\bar\xi^{(1)}-\sum_{i,j=1}^2c_{ij22}\,dz_i\wedge d\bar z_j
$$
and so
$$
\left(\frac{i}{2\pi}\Theta(\mathcal O_{X_1}(1))\right)^3=3!\left(\frac{i}{2\pi}\right)^3 D\,dz_1\wedge d\bar z_1\wedge\cdots\wedge d\xi^{(1)}\wedge d\bar\xi^{(1)},
$$
where we have set $D\overset{\text{\rm def}}=\lambda\mu$, which is, of course, a function $X_1\to\mathbb R$.
In particular, 
$$
\begin{aligned}
\int_{X_1}\left(\frac{i}{2\pi}\right)^3 D\,dz_1\wedge d\bar z_1\wedge\cdots & \wedge d\xi^{(1)}\wedge d\bar\xi^{(1)} \\ &=\frac 16\int_{X_1}\left(\frac{i}{2\pi}\Theta(\mathcal O_{X_1}(1))\right)^3\\ &= \frac 16(c_1^2(X)-c_2(X)),
\end{aligned}
$$
in fact, this integral over $X_1$ is just the top self-intersection of $c_1(\mathcal O_{X_1}(1))$, and this is easily seen to be $c_1^2(X)-c_2(X)$ by means of exact sequences (\ref{ses1}) and (\ref{ses2}).

Moreover, we have that
$$
\left(\frac{i}{2\pi}\right)^3 dz_1\wedge d\bar z_1\wedge\cdots\wedge d\xi^{(1)}\wedge d\bar\xi^{(1)}=\pi_{0,1}^*\left(\frac{1}{\pi^2}dV_\omega\right)\wedge\left(\frac{i}{2\pi}\Theta(\mathcal O_{X_1}(1))\right),
$$
so that
$$
\begin{aligned}
\int_{X_1}\left(\frac{i}{2\pi}\right)^3 dz_1\wedge d\bar z_1\wedge & \cdots\wedge d\xi^{(1)}\wedge d\bar\xi^{(1)} \\ &=\int_{X_1}\pi_{0,1}^*\left(\frac{1}{\pi^2}dV_\omega\right)\wedge\left(\frac{i}{2\pi}\Theta(\mathcal O_{X_1}(1))\right) \\
&= \frac 12 c_1^2(X)
\end{aligned}
$$
by Fubini. 

\subsection{A \lq\lq negative\rq\rq{} example: quotients of the ball}

Here, we wish to make an example to clarify why, if we deal with smooth metrics, we have to use the relatively nef weighted line bundles introduced above.

Suppose you want to show, using just $\mathcal O_{X_k}(1)$, the existence of global $k$-jet differentials on a surface $X$. From our point of view, a good possible \lq\lq test\rq\rq{} case is when $X$ is a compact unramified quotient of the unit ball $\mathbb B_2\subset\mathbb C^2$; surfaces which arise in this way are K\"ahler-Einstein, hyperbolic and with ample cotangent bundle: the best one can hope (these surfaces have even lots of symmetric differentials).

So, let $\mathbb B_2=\{z\in\mathbb C^2\mid |z|<1\}$ endowed with the Poincar\'e metric
$$
\begin{aligned}
\omega_P&=-\frac{i}{2}\partial\bar\partial\log(1-|z|^2)\\
&=\frac{i}{2}\biggl(\frac{dz\otimes d\bar z}{1-|z|^2}+\frac{|\langle dz,z\rangle|^2}{(1-|z|^2)^2}\biggr).
\end{aligned}
$$
Consider a  compact unramified quotient $X=\mathbb B_2/\Gamma$ with the quotient metric, say $\omega$.
Then, $\omega$ has constant curvature; in particular, the function $D\colon X_1\to\mathbb R$ we defined in \S\ref{D} is constant. 

This constant can be quite easily directly computed by hands. Here, we shall compute it as a very simple application of the celebrated Bogomolov-Miyaoka-Yau inequality $c_1^2\le 3\,c_2$ for surfaces of general type with ample canonical bundle, which says moreover that the equality holds if and only if the surface is a quotient of the ball $\mathbb B_2$.  

Using computations made in \S\ref{D}, we have
$$
\begin{aligned}
\frac{1}{6}\bigl(c_1(X)^2-c_2(X))&=\int_{X_1}\left(\frac{i}{2\pi}\right)^3 D\,dz_1\wedge d\bar z_1\wedge\cdots  \wedge d\xi^{(1)}\wedge d\bar\xi^{(1)}\\ &=D \int_{X_1}\left(\frac{i}{2\pi}\right)^3 dz_1\wedge d\bar z_1\wedge\cdots\wedge d\xi^{(1)}\wedge d\bar\xi^{(1)} \\
&=\frac{1}{2}\,c_1(X)^2\,D,
\end{aligned}
$$
so that, making the substitution $c_1(X)^2=3\,c_2(X)$, we find $D\equiv 2/9$. 

Now, a somewhat tedious computation of the $1$-index set of our curvatures, leads to the following result for the \lq\lq Morse\rq\rq{} integrals for $\mathcal O_{X_k}(1)$ and low values of $k$, using the new information about $D$.

\begin{itemize}

\item[$\boxed{k=1}$] We have already done this integral in \S\ref{D}: in this special case it gives $\frac 23\,c_1(X)^2>0$ and so the existence of $1$-jet differentials.

\item[$\boxed{k=2}$] In this case (the line bundle is no longer relatively nef) we don't have the equality $(X_2,\le 1)=X_2$ and so we have to determine the open set $(X_2,\le 1)$. This is an easy matter: using notations of \S\ref{surfnotation} and setting moreover $|v_1^1|^2=x$, $0\le x\le 1$, one sees from the expression of the curvature that
$$
(X_2,\le 1)=\biggl\{0<x<\frac 23\biggr\},
$$
since the trace of the \lq\lq horizontal\rq\rq{} part is always positive for $k=2$. Then we have
$$
\begin{aligned}
\int_{(X_2,\le 1)}& \left(\frac{i}{2\pi}\Theta(\mathcal O_{X_2}(1))\right)^4  \\
& =4!\left(\frac{i}{2\pi}\right)^4\int_{(X_2,\le 1)}(1-3x)\biggl(-\frac 13 x+\frac 29\biggr)\,dz_1\wedge\cdots\wedge d\bar\xi^{(2)}\\
&= 4!\left(\frac{i}{2\pi}\right)^3\int_{X_1}dz_1\wedge\cdots\wedge d\bar\xi^{(1)}\int_0^{2/3}(1-3x)\biggl(-\frac 13 x+\frac 29\biggr)\,dx\\
&=4!\left(\frac{i}{2\pi}\right)^3\int_{X_1}\frac 2{81}\,dz_1\wedge\cdots\wedge d\bar\xi^{(1)}\\
&=\frac 8{27}\,c_1(X)^2>0,
\end{aligned}
$$
where we make the substitution in the $\xi^{(2)}$-complex plane
$$
\frac{i}{2\pi}d\xi^{(2)}\wedge d\overline\xi^{(2)}\mapsto\frac{dx\,d\vartheta}{2\pi}.
$$
Hence the existence of $2$-jet differentials (the optimal attended result should be $10/27\,c_1(X)^2$, by replacing $c_2(X)=1/3\,c_1(X)^2$ in the expression of the leading term of the Euler characteristic $\chi(E_{2,m}T^*_X)$ of the bundle of invariant $2$-jet differentials on $X$, see \cite{Dem97}).

\item[$\boxed{k=3}$] Here the situation becomes much more involved. Several computations (which can be found in our PhD thesis \cite{DivPhD}) give
$$
\int_{X_3(\le 1,\mathcal O_{X_3}(1))}\biggl(\frac{i}{2\pi}\Theta(\mathcal O_{X_3}(1))\biggr)^5=-\underbrace{\frac{715933}{1944000}}_{\simeq\,0,37}\,c_1(X)^2<0,
$$
and so we are not able to check the existence of $3$-jet differentials by this method.
\end{itemize}
Here are some considerations. 

First, the value of the above integrals is, at least in these first cases, decreasing while morally one should expect an increasing sequence (the existence of $k$-jet differentials implies obviously the existence of $(k+1)$-jet differentials).

Second, we suspect that, in fact, this sequence continues to be non-increasing in general, since going up with $k$, adds more and more regions of negativity along the fiber direction ($\mathcal O_{X_k}(1)$ is not relatively positive over $X$, for $k\ge 2$). Moreover, recall that we are working here on a quotient of the ball, so that we had the most favorable \lq\lq horizontal\rq\rq{} contribution in terms of positivity: thus, the problem really relies in the fibers direction.

From these considerations, we deduce that to get a Green-Griffiths type result about asymptotic (on $k$) existence of section, we are naturally led to study either the smooth relatively nef line case (weighted line bundles $\mathcal O_{X_k}(\mathbf a)$), or to leave the \lq\lq smooth world\rq\rq{} and to study singular hermitian metrics on $\mathcal O_{X_k}(1)$ which reflects the relative base locus of this bundle.   

The rest of this paper will be devoted to the first of these two different approaches.

\subsection{Minimal surfaces of general type}

If we relax the hypothesis on the canonical bundle of the surface $X$, and we just take it to be big and nef, then our previous computation gives the same results. 

To see this, it suffices to select an ample class $A$ on $X$ and, for every $\varepsilon>0$, to solve the \lq\lq approximate\rq\rq{} K\"ahler-Einstein equation $\operatorname{Ricci}(\omega)=-\omega+\delta\,\Theta(A)$ (the existence of such a metric $\omega$ on $T_X$ is a well-known consequence of the theory of Monge-Amp\`ere equations).

Once we have such a metric we just observe that, with the notations of this section, we have $\lambda+\mu=-1+O(\delta)$, so that 
$$
\det(\alpha(c_{ij11})+\beta(c_{ij22}))=\alpha\beta(1+O(\delta))+\lambda\mu(\alpha^2-\beta^2)
$$
and
$$
\operatorname{tr}(\alpha(c_{ij11})+\beta(c_{ij22}))=-(\alpha+\beta)(1-O(\delta)).
$$
It is then clear that, our integral computation will now have a final error term which is in fact a $O(\delta)$, and thus we obtain the same results, by letting $\delta$ tend to zero.

\section{Proof of Theorem \ref{4.2}}

In this section we compute explicitly the Chern curvature of the weighted line bundles $\mathcal O_{X_k}(\mathbf a)$ on a surface $X$ and we find conditions for them to be relatively positive. Next, thanks to holomorphic Morse inequalities, we study the consequences of positive self-intersection and finally we prove Theorem \ref{4.2}.   

\subsection{Curvature of weighted line bundles}

We recall some notations and formulae. Let $v_s=v_s^1\,f_1^{(s)}+v_s^2\,f_2^{(s)}\in V_s$, with $|v_s^1|^2+|v_s^2|^2=1$ and set $x_s=|v^1_s|^2$, $0\le x_s\le 1$. Then, if
$$
R_s=\begin{pmatrix} 1-x_s & x_s \\ x_s & 1-x_s\end{pmatrix},\quad T=\begin{pmatrix} 1 & -1 \\ 0 & 1 \end{pmatrix}
$$
and
$$
R_p\cdot T\cdots R_q\cdot T=\begin{pmatrix} \delta_{p,q} & \gamma_{p,q} \\ \beta_{p,q} & \alpha_{p,q}\end{pmatrix},\quad p\ge q\ge 1,
$$
where $\alpha_{p,q},\beta_{p,q},\gamma_{p,q}$ and $\delta_{p,q}$ are functions of $(x_q,\dots,x_p)$, we have that, for $k\ge 2$,
$$
\begin{aligned}
\Theta(\mathcal O_{X_k}(1))=\,& d\xi^{(k)}\wedge d\bar\xi^{(k)}+\sum_{s=1}^{k-1}(\alpha_{k-1,s}-\beta_{k-1,s})\,d\xi^{(s)}\wedge d\bar\xi^{(s)} \\
& +\sum_{i,j=1}^2(-\beta_{k-1,1}\,c_{ij11}-\alpha_{k-1,1}\,c_{ij22})\,dz_i\wedge d\bar z_j.
\end{aligned}
$$
More generally, for $\mathbf a=(a_1,\dots,a_k)\in\mathbb Z^{k}$ (or possibly $\in\mathbb R^k$), we have

\begin{equation}\label{curvweight}
\begin{aligned}
\Theta(\mathcal O_{X_k}(a_1,\dots,a_k))&=a_k\,d\xi^{(k)}\wedge d\bar\xi^{(k)}+\sum_{s=1}^{k-1}\left(\sum_{j=s}^{k-1}a_{j+1}\,y_{j,s}+a_s\right)d\xi^{(s)}\wedge d\bar\xi^{(s)}\\
&\quad+\sum_{i,j=1}^2\left(\underbrace{\sum_{\ell=0}^{k-1}-a_{\ell+1}\beta_{\ell,1}\,c_{ij11}-a_{\ell+1}\alpha_{\ell,1}\,c_{ij22}}_{\overset{\text{\rm def}}=\,A_{ij}(a_1,\dots,a_k)}\right)dz_i\wedge d\bar z_j,
\end{aligned}
\end{equation}

where $y_{p,q}(x_q,\dots,x_p)\overset{\text{\rm def}}=\alpha_{p,q}-\beta_{p,q}$ (we also set formally $\alpha_{0,1}=\beta_{0,1}=1$). Observe that, for the $(2\times 2)$-matrix $(A_{ij})$, we have
$$
\tr(A_{ij})=\sum_{s=0}^{k-1}a_{s+1}\,w_{s,1},\quad w_{p,q}(x_q,\dots,x_p)\overset{\text{\rm def}}=\alpha_{p,q}+\beta_{p,q},
$$ 
thanks to the K\"ahler-Einstein assumption and formula (\ref{tr}).

Now, define $\theta_s^k=\theta_s^k(x_s,\dots,x_{k-1})$ to be the function given by
$$
(x_s,\dots,x_{k-1})\mapsto\sum_{j=s}^{k-1}a_{j+1}\,y_{j,s}+a_s.
$$
This is the $s$-th \lq\lq vertical\rq\rq{} eigenvalue of the weighted curvature. 

\begin{remark}\label{harmonic}
As the $\theta_s^k$'s are linear combinations of the $y_{j,s}$'s, we have that they all are of degree one in each variable. Hence they and their restriction to each edge of the cube $[0,1]^{k-s}$ are harmonic.
In particular they attain their minimum on some vertex of this cube.
\end{remark}

In $\mathbb R^k$, define the closed convex cone 
$$
\mathfrak N=\biggl\{\mathbf a\in\mathbb R^k\mid a_j\ge 2\sum_{\ell=j+1}^{k}a_\ell,\,\forall j=1,\dots,k-1\,\,\text{and}\,\, a_k\ge 0\biggr\}.
$$
We have the following three lemmas. 

\begin{lemma}\label{positiveeigenvalues}
The functions $\theta_s^k$ are positive if (and only if) $\mathbf a\in\overset\circ{\mathfrak N}$.
\end{lemma}

\begin{proof}
First of all, observe that the structure of the four functions $\alpha_{p,q}$, $\beta_{p,q}$, $\gamma_{p,q}$ and $\delta_{p,q}$ (and hence $y_{p,q}$) depends only on $p-q$. Now, it is immediate to check by induction that we have the following expression for the $\gamma_{p,q}$'s and the $\delta_{p,q}$'s:
$$
\gamma_{p,q}=-\sum_{h=q}^p\alpha_{h,q}\quad\text{\rm and}\quad\delta_{p,q}=1-\sum_{h=q}^p\beta_{h,q}.
$$
Next, observe that, for all $s\ge 1$,
$$
R_s(0)\cdot T=\begin{pmatrix} 1 & -1 \\ 0 & 1 \end{pmatrix}\quad\text{\rm and}\quad
R_s(1)\cdot T=\begin{pmatrix} 0 & 1 \\ 1 & -1 \end{pmatrix},
$$
so that, if $j\ge 1$, $y_{j+1,1}(\bullet,0)=y_{j,1}(\bullet)$ and $y_{j+1,1}(\bullet,1)=-1-2\,y_{j,1}(\bullet)-\sum_{h=1}^{j-1}y_{h,1}(\bullet)$; moreover, $y_{1,1}(0)=1$ and $y_{1,1}(1)=-2$.

The lemma is clearly true for $k=1$, so we proceed by induction on $k$. We have, for $s\ge 2$,
$$
\begin{aligned}
\theta^k_s & = \theta^k_s(x_s,\dots,x_{k-1};\mathbf a)= a_s+\sum_{j=s}^{k-1}a_{j+1}\,y_{j,s} \\
&=a_s+\sum_{j=s}^{k-1}a_{j+1}\,y_{j-s+1,1}=\theta^{k-s+1}_1(x_s,\dots,x_{k-1};\mathbf b),
\end{aligned}
$$
where $\mathbf b=(a_s,\dots,a_k)\in\mathbb R^{k-s+1}$ is again in the corresponding $\overset\circ{\mathfrak N}$: it remains then to show that, for a general $k\ge 2$, the lemma is true for $\theta^k_1$. Recall that, by Remark \ref{harmonic}, it suffices to check positivity on the vertices of the cube $[0,1]^{k-1}$. Let $\star$ denote an arbitrary sequence of $0$ and $1$ of length $k-2$: we shall treat the two cases $(\star,0)$ and $(\star,1)$ separately. For the first one, we have
$$
\begin{aligned}
\theta^k_1(\star\,,0;\mathbf a) &=a_1+\sum_{j=1}^{k-1}a_{j+1}\,y_{j,1}(\star,0) \\
&=a_1+\sum_{j=1}^{k-2}a_{j+1}\,y_{j,1}(\star)+a_k\,y_{k-1,1}(\star,0)\\
&=a_1+\sum_{j=1}^{k-2}a_{j+1}\,y_{j,1}(\star)+a_k\,y_{k-2,1}(\star)\\
&=a_1+\sum_{j=1}^{k-3}a_{j+1}\,y_{j,1}(\star)+(a_{k-1}+a_k)\,y_{k-2,1}(\star)\\
&=\theta^{k-1}_1(\star\,;\mathbf b')
\end{aligned}
$$
for a new $\mathbf b'\in\mathbb R^{k-1}$ which is easily seen to be in the corresponding $\overset\circ{\mathfrak N}$. Similarly, for the second case, we have
$$
\begin{aligned}
\theta^k_1(\star\,,1;\mathbf a) &=a_1+\sum_{j=1}^{k-1}a_{j+1}\,y_{j,1}(\star,1) \\
&=a_1+\sum_{j=1}^{k-2}a_{j+1}\,y_{j,1}(\star)+a_k\,y_{k-1,1}(\star,1)\\
&=a_1+\sum_{j=1}^{k-2}a_{j+1}\,y_{j,1}(\star)+a_k\left(-\sum_{h=1}^{k-3}y_{h,1}(\star)-2\,y_{k-2,1}(\star)-1\right)\\
\end{aligned}
$$
$$
\begin{aligned}
\qquad
&=(a_1-a_k)+\sum_{j=1}^{k-3}(a_{j+1}-a_k)\,y_{j,1}(\star)+(a_{k-1}-2a_k)\,y_{k-2,1}(\star)\\
&=\theta^{k-1}_1(\star\,;\mathbf b''),
\end{aligned}
$$
where again $\mathbf b''\in\mathbb R^{k-1}$ is a new weight which satisfies the (strict) inequalities defining $\mathfrak N$. The lemma is proved.
\end{proof}

The reason why we choose $\mathbf a$ in the interior of the cone $\mathfrak N$, is that with such a choice the vertical eigenvalues of the curvature $\Theta(\mathcal O_{X_k}(\mathbf a))$ are positive for all small $\varepsilon^{(k)}$.

\begin{remark}
The above lemma says in particular, that if $\mathbf a\in\mathfrak N$, then for all $\varepsilon>0$, we can endow $\mathcal O_{X_k}(\mathbf a)$ with a smooth hermitian metric $h_k$ (namely, the one we are working with) such that $\Theta_{h_k}(\mathcal O_{X_k}(\mathbf a))\ge-\varepsilon\,\omega$ along the fiber of $X_k\to X$, for some hermitian metric $\omega$ on $T_{X_k}$ (recall that we are always working modulo $\varepsilon^{(k)}$). In particular, the cone $\mathfrak N$ is contained in the cone of relatively nef (over $X$) line bundles.
\end{remark}

\begin{lemma}\label{positivetrace}
If $\theta_s^k\ge 0$ for all $s=1,\dots,k-1$, and $\mathbf a\in\mathbb N^{k}$ with at least one of the $a_j$'s is strictly positive, then $\tr(A_{ij})> 0$ in the cube $[0,1]^{k-1}$.
\end{lemma}

\begin{proof}
First of all, we recover the expression of the $w_{p,q}$'s in terms of the $y_{r,s}$'s. We have $w_{p,p}=(2+y_{p,p})/3$ and
$$
\begin{aligned}
w_{p,j-1}& =\alpha_{p,j-1}+\beta_{p,j-1}=x_{j-1}(\beta_{p,j}-\alpha_{p,j})+\alpha_{p,j} \\
& = \alpha_{p,j}-y_{p,j}\,x_{j-1}=\alpha_{p,j}+\frac{y_{p,j-1}+2\beta_{p,j}-\alpha_{p,j}}{3y_{p,j}}\,y_{p,j}\\
&=\frac{y_{p,j-1}+2w_{p,j}}{3},
\end{aligned}
$$
as, $x_{j-1}=-(y_{p,j-1}+2\beta_{p,j}-\alpha_{p,j})/3y_{p,j}$ (this is easily seen from the very definitions).
Then, by induction, we obtain
$$
w_{p,q}=\biggl(\frac 23\biggr)^{p-q+1}+\frac 13\sum_{\ell=q}^{p}\biggl(\frac 23\biggr)^{\ell-q}\,y_{p,\ell}.
$$
Now, $\tr(A_{ij})=\sum_{s=0}^{k-1}a_{s+1}\,w_{s,1}$ and so
$$
\begin{aligned}
\sum_{s=0}^{k-1}a_{s+1}\,w_{s,1} &= a_1+\sum_{s=1}^{k-1}\biggl(\biggl(\frac 23\biggr)^{s}+\frac 13\sum_{\ell=1}^{s}\biggl(\frac 23\biggr)^{\ell-1}\,y_{s,\ell}\biggr) \\
&= a_1+\sum_{s=1}^{k-1}\biggl(\frac 23\biggr)^s\,a_{s+1}+\frac 13\sum_{\ell,s=1}^{k-1}\biggl(\frac 23\biggr)^{\ell-1}a_{s+1}\,y_{s,\ell} \\
&=a_1+\sum_{s=1}^{k-1}\biggl(\frac 23\biggr)^s\,a_{s+1}+\frac 13\sum_{\ell=1}^{k-1}\biggl(\frac 23\biggr)^{\ell-1}(\theta_\ell^k-a_\ell) \\
&=\frac 23\,a_1+\biggl(\frac 23\biggr)^{k-1}a_k+\sum_{s=1}^{k-2}\biggl(\frac 23\biggr)^{s+1}a_{s+1}
+\frac 13\sum_{\ell=1}^{k-1}\biggl(\frac 23\biggr)^{\ell-1}\theta_\ell^k.
\end{aligned}
$$
\end{proof}

\begin{lemma}\label{positivedeterminant}
Let $D\colon X_1\to\mathbb R$ be as in \S\ref{D}. If $\theta_s^k\ge 0$ for all $s=1,\dots,k-1$, $\mathbf a\in\mathbb N^k$ with at least one of the $a_j$'s strictly positive and $D\equiv 2/9$, then $\det(A_{ij})>0$ in the cube $[0,1]^{k-1}$.
\end{lemma}

\begin{proof}
Set 
$$
\alpha(\mathbf a)\overset{\text{\rm def}}=a_1+\sum_{\ell=1}^{k-1}a_{\ell+1}\,\alpha_{\ell,1}\quad\text{and}\quad
\beta(\mathbf a)\overset{\text{\rm def}}=a_1+\sum_{\ell=1}^{k-1}a_{\ell+1}\,\beta_{\ell,1}.
$$
Then, formula (\ref{det}) yields 
$$
\begin{aligned}
\det(A_{ij})&=\alpha(\mathbf a)\beta(\mathbf a)+(\alpha(\mathbf a)-\beta(\mathbf a))^2D\\
&=\frac 14(\alpha(\mathbf a)+\beta(\mathbf a))^2-\frac 1{36} (\alpha(\mathbf a)-\beta(\mathbf a))^2
\end{aligned}
$$
But now, we observe that $\alpha(\mathbf a)+\beta(\mathbf a)=\tr(A_{ij})$ and that $\alpha(\mathbf a)-\beta(\mathbf a)=\theta^k_1$; the end of the proof of Lemma \ref{positivetrace} shows that $\tr(A_{ij})>1/3\,\theta^k_1$, so that $\det(A_{ij})>0$.
\end{proof}

\begin{proposition}\label{bigness}
If $\mathbf a\in\mathfrak N$, then the line bundle $\mathcal O_{X_k}(1)$ is big as soon as the top self-intersection $\mathcal O_{X_k}(\mathbf a)^{k+2}$ is positive.
\end{proposition}

\begin{proof}
Without loss of generality, we can suppose that $\mathbf a$ is integral and that $\mathbf a\in\overset\circ{\mathfrak N}$. Then Lemma \ref{positiveeigenvalues} ensures that all the \lq\lq vertical\rq\rq{} eigenvalues are positive: in this case, thanks to Lemma \ref{positivetrace}, we conclude that the curvature of $\mathcal O_{X_k}(\mathbf a)^{k+2}$ can have at most one negative \lq\lq horizontal\rq\rq{} eigenvalue. 

Thus, $X_k(\le 1,\mathcal O_{X_k}(\mathbf a))=X_k$ and so
$$
\begin{aligned}
\int_{X_k(\le 1,\mathcal O_{X_k}(\mathbf a))}\biggl(\frac i{2\pi}\Theta(\mathcal O_{X_k}(\mathbf a))\biggr)^{k+2}&=\int_{X_k}\biggl(\frac i{2\pi}\Theta(\mathcal O_{X_k}(\mathbf a))\biggr)^{k+2}\\ & \\ &=\mathcal O_{X_k}(\mathbf a)^{k+2}.
\end{aligned}
$$
If $\mathcal O_{X_k}(\mathbf a)^{k+2}>0$, then, by Demailly's holomorphic Morse inequalities, $\mathcal O_{X_k}(\mathbf a)$ is big and so is $\mathcal O_{X_k}(1)$ (recall that if $\mathbf a\in\mathbb N^k$, then there is a non-trivial morphism $\mathcal O_{X_k}(\mathbf a)\to\mathcal O_{X_k}(|\mathbf a|)$).
\end{proof}

\subsection{End of the proof}

Let $u_j=c_1(\mathcal O_{X_j}(1))$ be the first Chern class of the anti-tautological line bundle on $X_j$. Define the (real) polynomials $F_k,G_k\colon\mathbb R^k\to\mathbb R$ by
$$
(a_1\,u_1+\cdots+a_k\,u_k)^{k+2}=F_k(\mathbf a)\,c_1(X)^2-G_k(\mathbf a)\,c_2(X).
$$
Observe that these two polynomials do not depend on the particular surface $X$, but only on the relative structure of the fibration $X_k\to X$, which is universal.

\begin{lemma}\label{3FG}
Suppose that for each $k\ge 1$, there exists a minimal surface of general type $X$ such that $\mathcal O_{X_k}(1)$ is not big. Then, if $\mathbf a\in\mathfrak N$, we have the inequalities
$$
3F_k(\mathbf a)\ge G_k(\mathbf a)\ge 0
$$
and $G_k\not\equiv 0$.
\end{lemma}

\begin{proof}
Since $F_k$ and $G_k$ are independent of the particular surface chosen, we can suppose $X$ to be a compact unramified quotient of the ball $\mathbb B_2$. In this case, $D\equiv 2/9$ and, for $\mathbf a\in\overset\circ{\mathfrak N}$ rational,
$$
(a_1\,u_1+\cdots+a_k\,u_k)^{k+2}=\int_{X_k}\biggl(\frac i{2\pi}\Theta(\mathcal O_{X_k}(\mathbf a))\biggr)^{k+2}>0
$$
by Lemmas \ref{positiveeigenvalues}, \ref{positivetrace} and \ref{positivedeterminant}; on the other hand, by Bogomolov-Miyaoka-Yau, $c_1(X)^2=3\,c_2(X)$ and also $c_1(X)^2>0$.
Hence, by continuity, $3F_k(\mathbf a)-G_k(\mathbf a)\ge 0$ on $\mathfrak N$ (with strict inequality for $\mathbf a$ rational in the interior of the cone).

Now, let us compute the intersection $(a_1\,u_1+\cdots+a_k\,u_k)^{k+2}$ on a minimal surface of general type $X$ as in the hypotheses: of course, such a surface cannot be a compact unramified quotient of the ball $\mathbb B_2$. In this case we have $c_1(X)^2<3\,c_2(X)$, and so
$$
\begin{aligned}
(a_1\,u_1+\cdots+a_k\,u_k)^{k+2}&=F_k(\mathbf a)\,c_1(X)^2-G_k(\mathbf a)\,c_2(X)\\
&\ge\frac 13\,G_k(\mathbf a)\bigl(\underbrace{c_1(X)^2-3\,c_2(X)}_{<\,0}\bigr).
\end{aligned}
$$
Thus, if there exists a point $\mathbf a'\in\mathfrak N$ such that $G_{k}(\mathbf a')<0$, we would have $(a_1\,u_1+\cdots+a_k\,u_k)^{k+2}>0$ and hence, by Proposition \ref{bigness}, $\mathcal O_{X_k}(1)$ big, contradiction. Finally, if $G_k\equiv 0$, fix a rational point $\mathbf a$ in the interior of the cone $\mathfrak N$: such an $\mathbf a$ gives $F_k(\mathbf a)>0$. Then, for such a point we would have $(a_1\,u_1+\cdots+a_k\,u_k)^{k+2}=F_k(\mathbf a)\,c_1(X)^2>0$, again contradiction.
\end{proof}

\begin{remark}
Call the first hypothesis of the above lemma ($\sharp$). If ($\sharp$) is not satisfied, then we would have that there exist a $k\ge 1$ such that for each minimal surface of general type $X$, the line bundle $\mathcal O_{X_k}(1)$ is big.
In this case, we would already have the global sections we are looking for.
\end{remark}

Now, let $\Sigma_k\subset\mathbb R^k$ the zero locus of $G_k$. By the above lemma, $\mathfrak N\setminus\Sigma_k$ is dense in $\mathfrak N$. Set
$$
m_k\overset{\text{\rm def}}=\sup_{\mathbf a\in\mathfrak N\setminus \Sigma_k}\frac{F_k(\mathbf a)}{G_{k}(\mathbf a)}.
$$
If ($\sharp$) holds, then $m_k<+\infty$: otherwise, for each $M>0$ we would find an $\mathbf a_M\in\mathfrak N\setminus\Sigma_k$ such that $F_k(\mathbf a_M)>M\,G_k(\mathbf a_M)$ and so
$$
\begin{aligned}
F_k(\mathbf a_M)\,c_1(X)^2-G_k(\mathbf a_M)\,c_2(X) & >M\,G_k(\mathbf a_M)\,c_1(X)^2-G_k(\mathbf a_M)\,c_2(X) \\
&=G_k(\mathbf a_M)\bigl(M\,c_1(X)^2-c_2(X)\bigr).
\end{aligned}
$$ 
We would then contradict ($\sharp$), by choosing $M>c_2(X)/c_1(X)^2$. 

On the other hand, obviously, if $m_k>c_2(X)/c_1(X)^2$, then $\mathcal O_{X_k}(1)$ is big. 
Moreover, for each $k\ge 1$, we have $1/3\le m_k\le m_{k+1}$. The inequality $m_k\ge 1/3$ follows directly from Lemma \ref{3FG}. To see the monotonicity, notice that $(a_1\,u_1+\cdots+a_k\,u_k)^{k+2}|_{a_k=0}\equiv 0$ (just for dimension reasons) and so 
$$
(a_1\,u_1+\cdots+a_k\,u_k)^{k+2}=a_k\cdot\underbrace{\frac 1{a_k}(a_1\,u_1+\cdots+a_k\,u_k)^{k+2}}_{\text{well defined for $a_k=0$}}.
$$
But then,
$$
\begin{aligned}
\left.\frac 1{a_k}(a_1\,u_1+\cdots+a_k\,u_k)^{k+2}\right|_{a_k=0}&=\left.\frac{\partial}{\partial a_k}(a_1\,u_1+\cdots+a_k\,u_k)^{k+2}\right|_{a_k=0}\\
&=(k+2)(a_1\,u_1+\cdots+a_{k-1}\,u_{k-1})^{k+1}\cdot u_k \\
&=(k+2)(a_1\,u_1+\cdots+a_{k-1}\,u_{k-1})^{k+1},
\end{aligned}
$$
where the last equality is simply obtained by integrating along the fibers of $X_k\to X_{k-1}$. Hence, we have that
$$
\frac{F_k(a_1,\dots,a_{k-1},0)}{G_k(a_1,\dots,a_{k-1},0)}=\frac{F_{k-1}(a_1,\dots,a_{k-1})}{G_{k-1}(a_1,\dots,a_{k-1})}
$$
and monotonicity follows.

Finally, if we set $m_{\infty}$ to be the $\sup_{k\ge 1}m_k$, we find that, for $X$ a given minimal surface of general type, if $m_{\infty}>c_2(X)/c_1(X)^2$, then there exist a $k_0\in\mathbb N$ such that $\mathcal O_{X_{k_0}}(1)$ is big.

\begin{remark}
For the moment, we are not able to compute or even to estimate in a satisfactory way the limit term $m_{\infty}$. Of course, a divergent sequence would imply the existence of global invariant jet differentials of some order on every surface of general type. 
A less ambitious aim could be, for example, to encompass the case of hypersurfaces $X$ of $\mathbb P^3$ of degree greater than or equal to five (which is the minimum degree for $X$ to be of general type). 
In this case, a simple Chern classes computation shows that $m_\infty> 11$ would be sufficient.
\end{remark}

\subsubsection{Comparison with lower bounds of \cite{Div08}}\label{lb}

For low values of $k$, one can compute directly the intersection product 
$$
(a_1\,u_1+\cdots+a_k\,u_k)^{k+2}
$$ 
either algebraically, by means of sequences (\ref{ses1}) and (\ref{ses2}), or using our curvature formula and computing the corresponding integrals (for more details, see \cite{DivPhD}), on some particular $k$-tuple $\mathbf a$. For instance, a natural choice is given by the sequences $(2,1)$, $(6,2,1)$, $(18,6,2,1)$ and so on. These supply the estimates
$$
m_1\ge 1,\quad m_2\ge\frac{13}{9}\simeq 1,44 ,\quad m_3\ge \frac{1195}{742}\simeq 1,61,\quad m_4\ge\frac{442243}{271697}\simeq 1,63,
$$
which give the existence of global invariant jet differentials on a minimal surface of general type whose Chern classes satisfy the following inequalities:
$$
\begin{matrix}
\text{order 1 jet differentials if} & c_1(X)^2>c_2(X),\\
\text{order 2 jet differentials if} & 13\,c_1(X)^2>9\,c_2(X), \\
\text{order 3 jet differentials if} & 1195\,c_1(X)^2>742\,c_2(X), \\
\text{order 4 jet differentials if} & 442243\,c_1(X)^2>271697\,c_2(X).
\end{matrix}
$$  
Unfortunately, these first terms are still very far from being close even to $11$. 

We would like to remark here, that, even if we are dealing with the same relatively nef bundles of \cite{Div08}, we get considerably better lower bounds for the degree in the case of hypersurface in $\mathbb P^3$.

The reason is quite subtle: from a hermitian point of view, in proving Theorem 2 of \cite{Div08}, we tacitly used the restriction of the Fubini-Study metric of the projective space the hypersurface is embedded in, to the tangent bundle of the hypersurfaces. This is why he had to \lq\lq correct\rq\rq{} this metric by adding some positivity coming from $\mathcal O(2)$ and thus loosing some effectivity.

Using instead the differential-geometric approach of the present paper, we were able to take advantage of the full strength of the K\"ahler-Einstein metric, which reflects directly the strong positivity properties of varieties with ample canonical bundle. 

\bibliography{MyBib}{}

\end{document}